\documentclass[10pt]{amsart}
\usepackage{amssymb}

\newtheorem{theorem}{Theorem}

\newtheorem{proposition}[theorem]{Proposition}

\newtheorem{definition}[theorem]{Definition}
\newtheorem{remark}[theorem]{Remark}

\newcommand{\vol}{{\operatorname{Vol}}}

\renewcommand{\phi}{\varphi}

\newtheorem{theo}{{\sc Theorem}}

\newtheorem{lem}[theo]{{\sc Lemma}}

\newenvironment{defn}{\medskip\noindent{\it Definition:\/} }{\medskip}

\begin{document}

\title[$L^{2}$-restriction bounds for eigenfunctions along curves in quantum completely integrable case ]
{$L^{2}$-restriction bounds for eigenfunctions along curves in the quantum completely integrable case}
\thanks{The author was supported by a William Dawson Fellowship and
 NSERC Grant OGP0170280}
\author{John A. Toth}

\address{Department of Mathematics and Statistics, McGill University, Montreal,
Canada}

\email{jtoth@math.mcgill.ca}

\maketitle

\begin{abstract}We show that for a quantum completely integrable system in two dimensions,the $L^{2}$-normalized joint eigenfunctions of the commuting semiclassical pseudodifferential operators satisfy restriction bounds of the form $ \int_{\gamma} |\phi_{j}^{\hbar}|^2 ds = {\mathcal O}(|\log \hbar|)$ for  {\em generic} curves $\gamma$ on the surface. We also prove that the maximal restriction bounds of Burq-Gerard-Tzvetkov \cite{BGT} are always attained for certain exceptional subsequences of eigenfunctions.

\end{abstract}

\section{Introduction}\label{section1}
Let $(M,g)$ be a compact, closed orientable Riemannian manifold
 Let $-\Delta : C^{\infty} (M) \rightarrow C^{\infty}(M)$ 
 be the associated Laplace-Beltrami operator eigenvalues
  $0 < \lambda_{1} \leq \lambda_{2} \leq \cdots$ and
   eigenfunctions $\phi_{j}; j=1,2,3,...$ satisfying
  $$ -\Delta_{g} \phi_{j} = \lambda_{j}^{2} \phi_{j},$$
   and $L^{2}$-normalized so that $\int_{M} |\phi_{j}|^{2} dvol(x) =1$. The celebrated Avakumovic-Levitan-H\"{o}rmander formula implies that 
 \begin{equation} \label{Hormander}
\| \phi_{j} \|_{L^{\infty}} = {\mathcal O}(\lambda_j^{\frac{n-1}{2}}).  
\end{equation}
  The example of the sphere shows that (\ref{Hormander}) is sharp. The corresponding sharp $L^{p}$-bounds are due to Sogge \cite{So1,So2,Soggebook}. Even though this $L^{\infty}$-bound  is far from generic \cite{STZ}, the only general improvements on (\ref{Hormander}) that we are aware of are due to Sogge and Zelditch \cite{SZ} and more recently, Sogge, Toth and Zelditch \cite{STZ, T4}.  These authors  obtain pointwise $o(\lambda^{\frac{n-1}{2}})$-bounds under a certain non-recurrence condition for the geodesic flow on $(M,g)$. The methods in \cite{STZ} follow closely the earlier work of Safarov \cite{S} and Safarov-Vassiliev \cite{SV}. 
  
     It is natural to ask whether one can generically improve the ${\mathcal O}(\lambda^{\frac{n-1}{2} })$- sup-bound by polynomial powers of $\lambda$ and if so, by how much?  In general, very little is known here: Polynomial improvements have been obtained by Iwaniec and Sarnak in  arithmetic hyperbolic cases \cite{Sa,IS}.  At the other extreme, in the {\em quantum completely integrable} (QCI) case  it is known that  under a natural Morse assumption,
   one can show that $\sup_{x \in M} |\phi_{\lambda}(x)| = {\mathcal O}(\lambda^{\frac{1}{4}})$ when the $\phi_{\lambda}$'s are joint eigenfunctions of the commuting operators and  $\dim M = 2$ (see \cite{T4}).   In the latter case, when $\dim M >2$,  one can at least hope to obtain a fairly complete answer to this question provided   the $\phi_{j}$'s are joint eigenfunctions of $n$-functionally independent, self-adjoint, joint elliptic, commuting $\hbar$-pseudodifferential operators $P_{1}(\hbar), P_{2}(\hbar),...,P_{n}(\hbar)$. However,  due to the presence of often complicated degeneracies of the Lagrangian foliation, even at the classical level the dynamics is only partially understood in general \cite{VN2}.  Similarily, at the quantum level, the  asymptotic  blow-up properties of eigenfunctions  (eg. sharp $L^{p}$-bounds) are also only partially understood \cite{T1}-\cite{T3},\cite{TZ1}-\cite{TZ3}.

  Apart from pointwise bounds, it is natural when studying
   asymptotic concentration properties of eigenfunctions to consider limits of expected values $\langle A \phi_{\lambda}, \phi_{\lambda} \rangle$ as $\lambda \rightarrow \infty$ and to compute the corresponding semiclassical defect measures. Formally, one can let $A$ approach $\delta_{\gamma}$, where the latter is surface measure along a submanifold, $\gamma \subset M$. Then, one is faced with  estimating  asymptotic upper bounds for $L^{2}$, or more generally, $L^{p}$ integrals along submanifolds of $M$. In the case of surfaces, these are curves and the concentration of these defectmeasures along  a periodic geodesic $\gamma$ is called {\em strong scarring}. For Laplace eigenfunctions, the eigenfunction restriction bounds  have been studied by Reznikov \cite{R} for hyperbolic surfaces and  Burq-Gerard-Tzvetkov \cite{BGT} for general manifolds (and for all $p \geq 2$).  Both papers are related to earlier  work of Tataru \cite{Ta}  on estimating boundary traces of wavefunctions.
 We will focus here on the case where $p=2$ and $\dim M = 2$. (ie. $L^{2}$- restriction bounds along curves on surfaces).  At the moment, it is unclear to us whether our methods extend to $L^{p}$-restriction bounds for $p \neq 2$.  In the special case of $L^{2}$-integrals along curves, the estimates in \cite{BGT} are as follows:
  \begin{itemize} \label{bgt}
 \item (i) If $\gamma$ is a unit-length geodesic, then
 $$ \int_{\gamma} |\phi_{j}(s)|^{2} ds = {\mathcal O}(\lambda_j^{1/2}).$$
 \item(ii) If $\gamma$ is a curve with strictly-positive geodesic curvature,
 $$ \int_{\gamma} |\phi_{j}(s)|^{2} ds = {\mathcal O}(\lambda_j^{1/3}).$$ 
  \end{itemize}

In this article, we obtain {\em generic} asymptotic bounds
   for $\int_{\gamma}|\phi_{j}(s)|^{2} ds$ in the case where the $\phi_{j}$'s are {\em joint} eigenfunctions of the QCI system consisting of two commuting $\hbar$-pseudodifferential operators $P_{1}(\hbar)$ and $P_{2}(\hbar)$.     
  Other than the fact  that our analysis here is specific to QCI systems and to the case $p=2$, this paper differs from \cite{BGT} in several ways: 
  
  1) One of the main issues here is the {\em generic} behaviour of restriction bounds, where a curve $\gamma:[a,b] \rightarrow M$ is called generic if it satisfies the Morse condition in \ref{generic}. As we show in section 2, in the QCI case the restricted asymptotic eigenfunction mass, $\int_{\gamma} |\phi_j|^{2} ds,$  is much smaller than the prediction in (i) or (ii) above.  Indeed,  it is ${\mathcal O}(\log \lambda_j)$  (see Theorems \ref{mainthm} and \ref{mainthm2}) and the example of zonal harmonics on the sphere (see section \ref{zonal}) shows that this estimate is sharp.
  
  2) In Theorem \ref{mainthm2}, we establish a converse to (i) above, and show that the bound in (i) is {\em always} attained in the QCI case. Moreover, we identify the specific bicharacteristics in terms of the singular Lagrangian foliation that support such large eigenfunction scars.
  
  3) Finally, we prove all our results for a rather large class of possibly inhomogeneous semiclassical QCI Hamiltonians. The semiclassical Laplacian $P_{1}(\hbar) = -\hbar^{2} \Delta$ is a special case.   The results really have to do with
   the bicharacteristic flow and are not specific to geodesics.
   
Before going on, we explain what is meant here by the term {\em generic}.
Given $E_{1} >0$ a regular value of $p_{1}$, we assume that
for $(x,\xi)  \in p_{1}^{-1}(E_{1}),$ 
$$\partial_{\xi} p_{1}(x,\xi) \neq 0. \,\,\,\,(A1)$$
That is,  $p_{1}$ is real prinicipal type on the hypersurface $p_{1}^{-1}(E_1)$.
We define
 \begin{equation} \label{cylinder}
   C_{\gamma}:= \{ (x,\xi) \in T^{*}M; \, p_{1}(x,\xi) =E_{1}, \,\,\, x \in \gamma \} = p_{1}^{-1}(E_{1}) \cap \pi^{-1}(\gamma).
   \end{equation}
  
\begin{definition} \label{generic} Let $\iota: C_{\gamma} \rightarrow p_{1}^{-1}(E_1)$ be the standard  inclusion map.
We say that the ${\mathbb R}^{2}$-integrable system with moment map ${\mathcal P} = (p_{1}, p_{2})$ is {\em generic} along the curve
$\gamma: [a,b] \rightarrow M$ provided $\iota^{*}p_{2} \in  C^{\infty}( C_{\gamma} )$   is a Morse function and condition (A1) is satisfied.
\end{definition}

\begin{remark}
In the homogeneous case where $p_{1} = |\xi|^{2}_{g}$, the manifold $C_{\gamma} = S^{*}_{\gamma}M $. In this case, $E_{1} =1$ and by Euler homogeneity, $  \xi \cdot \partial_{\xi} |\xi|^{2}_{g} = 2 |\xi|^{2}_{g}$ so that (A1) is automatically satisfied.
\end{remark}
  \begin{theo} \label{mainthm}
Let $\phi^{\hbar}_{j}; j=1,2,3,...$ be the $L^{2}$-normalized joint eigenfunctions of the commuting operators $P_{1}(\hbar)$ and $P_{2}(\hbar)$ on a Riemannian surface $(M^{2},g)$ with joint eigenvalues $ (\lambda_{j}^{(1)}(\hbar) = E_{1} + {\mathcal O}(\hbar), \, \lambda_{j}^{(2)}(\hbar) ) \in Spec P_{1}(\hbar) \times Spec P_{2}(\hbar) ;  \, j=1,2,3,...$ Then for {\em generic} curves $\gamma: [a,b] \rightarrow M$ and $\hbar \in (0,\hbar_{0}]$,
$$ \int_{\gamma} |\phi^{\hbar}_{j}|^{2} ds = {\mathcal O}_{|\gamma|} \left(    |\log \hbar| \right). $$
Here, $|\gamma |$ denotes the length of the curve segment  $\gamma$ and the RHS of the above estimate is uniform over all energy values  $\{E \in {\mathbb R}; \,  (E_{1},E)  \in  {\mathcal P}(T^{*}M) \}$.
\end{theo}
In the special case where $P_{1}(\hbar) = -\hbar^{2} \Delta$ one can of course scale out $\hbar$ and Theorem \ref{mainthm} becomes
\vspace{2mm}

\begin{theo} \label{corollary}
Let $\phi_{j}; j=1,2,3,...$ be the $L^{2}$-normalized joint Laplace eigenfunctions of the commuting operators $P_{1}= -\Delta $ and $P_{2}$ on a Riemannian surface $(M^{2},g)$. Then, provided $\iota^{*}p_{2}|_{S^*_{\gamma}M}$ is Morse, one gets
$$ \int_{\gamma} |\phi_{j}|^{2} ds = {\mathcal O}_{|\gamma|} \left( \log \lambda_j \right). $$
\end{theo}

Just as in the case of maximal $L^{\infty}$-bounds, it turns out  zonal harmonics on spheres of revolution saturate the bounds in Theorems \ref{mainthm} and \ref{corollary}, so they are sharp. We discuss this example in detail in section \ref{examples}. 

 We will show in section \ref{nongeneric} (see Proposition \ref{mainprop}) that (certain) bicharacteristics  are  non-generic curves in the sense of Definition \ref{generic}.  In the homogeneous case,  it was already observed in \cite{BGT} (see estimate (i) above) that in the case where $\gamma$  is  a  geodesic, the restriction upper bounds can grow at the maximal rate $\sim \lambda^{1/2}$  Consistent with this,     in the QCI case, we will  show that there {\em always} exist certain  bicharacteristics  that support high $L^{2}$-mass for certain subsequences of eigenfunctions consistent with the $\lambda^{1/2}$-bound in  (i) (at least up to possible loss of $\log \lambda$). However,  it is important to note that  the {\em nature} of the bicharacteristic is very important when discussing restriction bounds.  To describe what we mean,  let ${\mathcal B}_{reg}$ (resp. ${\mathcal B}_{sing}$) denote the regular (resp. singular) values of the moment map ${\mathcal P}= (p_{1}, p_{2}) \subset {\mathbb R}^{2}$. In the general QCI case, most bicharcteristics of $H_{p_1}$ are subsets of Lagrangian tori in ${\mathcal P}^{-1}(B_{reg})$. These do not support large $L^{2}$-bounds along their configuration space projections. However, as was shown in \cite{TZ3} Lemma 3, unless $(M,g)$ is a flat torus, there is {\em always}  a subsequence of joint eigenfunctions of $P_1$ and $P_{2}$ with mass concentrated along (singular) joint orbits of the Hamilton fields $H_{p_1}$ and $H_{p_2}$ contained in ${\mathcal P}^{-1}({\mathcal B}_{sing})$.  The latter eigenfunctions saturate the maximal bounds in (ii) above.  This is of course consistent with simple examples like surfaces of revolution with metric $g = dr^{2} + a^{2}(r) d\theta^{2}$, where the equator is the projection of a singular orbit of the joint flow of $H_{p_{1}}$ and $H_{p_{2}}$. In this case, $p_{1} = |\xi|_{g}$ and $p_{2} = p_{\theta}$ with $p_{\theta}(v):= \langle v, \partial_{\theta} \rangle$. The  corresponding joint eigenfunctions, $\phi_{j}^{\hbar}$, of $P_{1}(\hbar) = -\hbar^{2} \Delta_{g}$ and $P_{2}(\hbar) = \hbar D_{\theta}$ with joint eigenvalues $(\lambda^{(1)}_{j}(\hbar), \lambda^{(2)}_{j}(\hbar)) = (1,0) + o(1)$  (the analogs of highest weight spherical harmonics) satisfy  $\int_{\gamma} |\phi_{j}^{\hbar}|^{2} ds \sim \hbar^{-1/2}$ along the equator, $\gamma$. So, in particular, the  equatorial geodesic is certainly non-generic in the sense of  Definition \ref{generic}. However, it is not hard to show that (see section \ref{examples}) the meridian great circles, while obviously also periodic geodesics {\em are} generic. This geodesic lies in the base space projection of a maximal Lagrangian torus. The zonal harmonics have $\hbar$-microsupport on this torus and, as we show in subsection \ref{zonal} that the latter eigenfunctions  have $L^{2}$-restriction bound $\int_{\Gamma} |\phi_j^{\hbar}|^{2} \sim \log \frac{1}{\hbar}$ along any meridian great circle, $\Gamma$. This is consistent with the results of Theorems \ref{mainthm} and \ref{corollary}.
In the case of exceptional, non-generic bicharacteristics, we prove

\begin{theo} \label{mainthm2}
Let $P_j;(\hbar);  j=1,2$ be an Eliasson non-degenerate, QCI system on a surface, $(M,g)$. Then,
\begin{itemize}
\item (i) \, When $\gamma$ is the projection of a bicharacteristic segment of $p_{1}$  contained in $ {\mathcal P}^{-1}({\mathcal B}_{reg}),$
$$ \int_{\gamma}|\phi_{j}(s;\hbar)|^{2} ds = {\mathcal O}_{|\gamma|}(1),$$ 
\item (ii) \, When $\gamma$ is the projection of a singular joint orbit in ${\mathcal P}^{-1}({\mathcal B}_{sing}),$ 
$$ \int_{\gamma} |\phi_{j}(s;\hbar)|^{2} ds = {\mathcal O}_{|\gamma|}(\hbar^{-1/2} ).  $$
Moreover,
 there exists a constant $c_{\gamma} >0$ depending only on the curve $\gamma,$ and a subsequence of joint eigenfunctions, $\phi_{j_k}^{\hbar}; k=1,2,...$ such that for $\hbar \in (0,\hbar_0],$
 $$ \int_{\gamma} |\phi_{j_k}(s;\hbar)|^{2} ds \geq  c_{\gamma} \hbar^{-1/2}  \,\,\, \mbox{when} \, \gamma \,\mbox{ is stable},$$
$$ \int_{\gamma} |\phi_{j_k}(s;\hbar)|^{2} ds \geq c_{\gamma} \hbar^{-1/2} |\log \hbar|^{-1}\,\,\, \mbox{when} \, \gamma \, \mbox{is unstable}.$$
\end{itemize}
\end{theo}

 It is proved in \cite{TZ3} (see Lemma 3) that unless $(M,g)$ is a flat torus, the joint  flow $\Phi^{t}$ always possesses at least one singular orbit (see also \cite{L, LS}).  In the case where $\dim M =2$ this orbit must  be one-dimensional (ie. a geodesic). Thus, the second estimate (ii) in Theorem \ref{mainthm2} is always attained in the QCI case and therefore, up to a power of $\log \lambda$,  the maximal $L^{2}$-restriction bound in \cite{BGT} is always attained.
\begin{remark}
Examples to which Theorems \ref{mainthm} and \ref{mainthm2} apply include: QCI Laplacians on ellipsoids (with distinct axes), surfaces of revolution, Liouville surfaces. Less well-known examples include QCI Laplacians associated with spherical metrics got by reducing the Goryachev-Chaplyin top as well as those constructed in \cite{DM}. In both of the last two classes of examples, the integral in involution, $p_{2}$, is a {\em cubic} polynomial in the momentum variables. Finally, there are also examples known where $p_{2}$ is {\em quartic} in the momenta \cite{Mi}. In addition, our results apply to  inhomogeneous QCI systems such as Neumann oscillators, Euler and  Kowalevsky tops and  the spherical pendulum as well as many others.
 We have stated our results for surfaces because the formulation is quite elegant in that case. It is not hard to extend the  analysis  here to higher-dimensions under the appropriate notion of a generic submanifold,  but the formulation of results becomes more cumbersome. We hope to address this elsewhere.
\end{remark}

\begin{remark}
 In analogy with the specific results for QCI eigenfunctions in Theorems \ref{mainthm} and \ref{mainthm2} above, it is natural to try to determine $L^{2}$ (or $L^{p}$)  eigenfunction bounds along  ``typical'' curves on general Riemann surfaces $(M^{2},g)$ by varying the standard restriction estimates over  appropriate moduli spaces of curve segments. We hope to address this point elsewhere.
\end{remark}

We thank Steve Zelditch for helpful comments and suggestions regarding an earlier version of
 the manuscript.
 
\section{Generic (joint) eigenfunction restriction bounds along curves}
 We say that $P(\hbar) \in Op_{\hbar, cl}(S^{m}_{k}(T^*M))$ if locally it has Schwartz kernel
 $$P(x,y;\hbar) = (2\pi \hbar)^{-n} \int_{{\mathbb R}^{n}} e^{i(x-y)\xi/\hbar} p(x,\xi;\hbar)  d\xi$$
 where $a(x,\xi;\hbar) \sim \sum_{j=0}^{\infty} a_{j}(x,\xi) \hbar^{-k+j}$ and
 $a_{j} \in S^{m-j}(T^{*}M)$ with $\hbar \in (0,\hbar_0]$. From now on, without loss of generality we assume that $P_{j}(\hbar) \in Op_{\hbar, cl}(S^{m}_{0})$ and that $P_{1}(\hbar)^{2} + P_{2}(\hbar)^{2}$ is elliptic in the classical sense and the $P_{j}(\hbar)$'s are self-adjoint.
 
  In this section, we get {\em generic} asymptotic bounds
   for $\int_{\gamma}|\phi_{j}(s;\hbar)|^{2} ds$ in the case where the $\phi_{j}$'s are {\em joint} eigenfunctions of $P_{1}(\hbar)$ and $P_{2}(\hbar)$. Here, the term {\em generic} refers to a non-degeneracy condition on the QCI system along the (generalized) cylinder $C_{\gamma}$ given in Lemma \ref{generic}. 
   
\subsection{Proof of Theorem \ref{mainthm}.}
   
We assume here that $(M,g)$ is a compact surface with
QCI quantum Hamiltonian given by  $P_{1}(\hbar)$ 
and the quantum integral in involution is $P_{2}(\hbar)$ where we assume that
its  principal symbol, $p_{2}$, satisfies the Morse condition in (A1). 
 The joint spectrum of $P_1(\hbar)$(resp. $P_{2}(\hbar)$) will be
denoted by $\lambda^{(1)}_j(\hbar)$ (resp. $\lambda^{(2)}_j(\hbar)$) with
$j=1,2,3,...$.
 Let $\rho  \in {\mathcal S}({\mathbb R})$ satisfy $\rho(u) \geq 0$ with $\rho(0)=1$ and $\hat{\rho} \in C^{\infty}_{0}([-\epsilon, \epsilon])$ with $\epsilon >0$ sufficiently small.  For fixed $x\in M$, we  form  the joint unintegrated trace attached to the level $(p_1, p_{2}) = (E_1,E_2)$ given by
\begin{equation} \label{qci1}
I_{E}(x;\hbar):= \sum_{j=1}^{\infty} \rho(
\hbar^{-1}[\lambda^{(1)}_j(\hbar) -E_1]) \,\rho( \hbar^{-1}[\lambda^{(2)}_j(\hbar)
- E]) \, | \phi_{j}(x;\hbar)|^{2}.
\end{equation}
Our task is to obtain a locally uniform asymptotic bound (in $E$)
for $\int_{a}^{b} I_{E}(x(\tau);\hbar) \, d\tau$  as $\hbar \rightarrow 0$ . Writing the
usual small-time $\hbar$-Fourier integral operator (FIO)  parametrices for $e^{itP_{1}(\hbar)}$ and
$e^{isP_{2}(\hbar)}$, and taking Fourier transforms in (\ref{qci1})
gives:
\begin{eqnarray} \label{qci2}
\,\,\,\,\,\,\,\,\,\,\,\,\,\,I_{E}(x;\hbar) = (2\pi \hbar)^{-4} \int \int \int \int \int e^{i \Phi(x,y,\xi,\eta,s;E)/\hbar} a_{\chi}(y,\eta,\xi;\hbar) \, \hat{\rho}(t)  \hat{\rho}(s) \, ds dt d\xi d\eta dy   \nonumber \\
+ {\mathcal O}(1).
\end{eqnarray}
In (\ref{qci2}),
\begin{equation} \label{phase1}
\Phi(x,y,\xi,\eta,s,t;E) = \psi_1(x,y,\xi,t) + tE_1 + \psi_2(y,x,\eta,s) +sE,
\end{equation}
where,
\begin{equation} \label{phase1}
\psi_1(x,y,\xi,t) = \phi_1(x,\xi,t) - y \xi, \,\,\, \psi_2(y,x,\eta,s) = \phi_2(y,\eta,s) -x \eta.
\end{equation}
 In (\ref{phase1}), the $\phi_{j}; j=1,2$ satisfy the usual eikonal initial value problems
 \begin{equation} \label{eikonal} 
 \partial_{t} \phi_{1} + p_{1}(x, \partial_x\phi_1) = 0, \,\,\, \phi_1|_{t=0} = x \xi, \end{equation}
 $$  \partial_{s} \phi_{2} + p_{2}(y, \partial_y\phi_2) = 0, \,\,\, \phi_2|_{s=0} = y \eta.$$
 From the equations in (\ref{eikonal}) one easily derives the following Taylor expansions for $\phi_1(t,x,\xi) $(resp. $\phi_2(s,y,\eta)$)  centered at $t=0$ (resp. $s=0$):
 \begin{equation} \label{phaseformula1}
 \phi_{1}(t,x,\xi) =  x \xi - t p_{1}(x,\xi)  + {\mathcal O}(t^{2}).
 \end{equation}
  \begin{equation} \label{phaseformula2}
 \phi_{2}(s,y,\eta) = y \eta - s p_{2}(y,\eta)  + {\mathcal O}(s^{2}).
 \end{equation}
 In (\ref{qci2}) the amplitude is of the form
\begin{equation} \label{amplitude1}
 a_{\chi} (y,x,\eta,\xi,s,t;\hbar)
= \chi(E_1-p_{1}(x,\xi))  \chi(E-p_{2}(y,\eta)) \chi(y - x)
a(x,y,\eta,\xi,s,t;\hbar),
\end{equation}
where, $ a \sim \sum_{j=0}^{\infty} a_{j} \hbar^{j}$, $a_{0} \geq
\frac{1}{C_{0}} >0$ and $\chi \in C^{\infty}_{0}(R)$ with $\chi =1$
near the origin.  


 Since the integral in (\ref{qci2}) is absolutely convergent, we carry out the $(y,\eta)$-integration first and get that
\begin{equation} \label{qci5}
I_{E}(x;\hbar) = (2\pi \hbar)^{-4} \int \int \int  \exp  \, [ it (E_1-p_{1})(x,\xi) + {\mathcal O}(t^2) /\hbar ] 
\end{equation}
 $$ \times \, \hat{\rho}(t) \, I(s,t,x,\xi;\hbar) \, ds d\xi dt,$$
where,
\begin{equation} \label{fubini}
 I(s,t,x,\xi;\hbar):= \int \int
e^{i\Phi(y,\eta;x,s)/\hbar} b_{\chi}(y,\eta,x,\xi,s,t;\hbar) 
 \hat{\rho}(s)
dy d\eta,
\end{equation}
and where $b \in S^{0}(1)$ with $b \sim \sum_j b_{j} \hbar^{j}, \,\,\, b_0 \geq 1/C_{0}>0$
and $b_{\chi}$ has the same properties  as $a_{\chi}$ in (\ref{amplitude1}).
The phase function
\begin{equation} \label{qci6}
\Phi(y,\eta;x,s) =  \langle x-y, \xi - \eta \rangle + s(E-p_{2}(y,\eta)) + {\mathcal O}_{y,\eta}(s^{2}).
\end{equation}

Since $ \det (\Phi''_{y,\eta}) = 1 + {\mathcal O}(s)$ and the $s$-support of $b_{\chi}$ can be taken arbitrarily small, one can apply stationary phase (with parameters) in the $(y,\eta)$-variables  in (\ref{fubini}). The critical point equations for $(y,\eta)$ are 
$$ \eta = \xi + s  \, \partial_{y}p_{2}(y,\eta)  + {\mathcal O}(s^{2}) \,\,\,\,\,\,\, (*)$$
$$ y = x + s \, \partial_{\eta}p_{2}(y,\eta) + {\mathcal O}(s^{2}) .$$
 By a straightforward computation,   
$I_{E}(x,\hbar)$ equals
\begin{equation} \label{qci7}
 (2\pi \hbar)^{-2} \int \int \int \exp \, [i
t(1-p_{1})(x,\xi) + s(E- p_{2})(x,\xi)   +  {\mathcal
O}_{x,\xi}(s^{2}) + {\mathcal O}_{x,\xi}(t^{2})/\hbar ]  \,
\end{equation}
$$ \times c(x,\xi,s,t;\hbar) d\xi dt ds   + {\mathcal O}(1),$$
where, $c \in S^{0}_{cl}(1)$ with $c(x,\xi,s,t;\hbar) \sim
\sum_{j=0}^{\infty} c_{j}(t,s,\xi) \hbar^{j}$ where the $c_{j} \in C^{\infty}_{0}$.

Next, we make a polar variables decompostion in
the $\xi$-variables in (\ref{qci7}), which is legitimate since by assumption, $p_{1}$ is real principal type on the energy shell $p_{1}^{-1}(E_{1})$ and so,  $|\partial_{\xi} p_{1}| \geq \frac{1}{C}>0$ when $p_{1} \sim E_{1}$ and  supp $c \subset [-\epsilon, \epsilon]^{2} \times p_{1}^{-1}[E_{1}-\epsilon, E_{1}+\epsilon]$. We note that in the case of a Schrodinger operator, $p_{1}(x,\xi) = |\xi|_{g}^{2} + V(x)$ and so, $2\xi \partial_{\xi} p_{1} = |\xi|_{g}^{2}$ by Euler homogeneity. So, as long as 
$$\gamma \cap \{ x \in M; V(x) = E_1 \} = \emptyset \,\,\,\,(*)$$
 the condition $\partial_{\xi} p_{1}(x,\xi) \neq 0$ is satisfied for $(x,\xi) \in p_{1}^{-1}(E_1) \cap \pi^{-1}(\gamma)$. In the homogeneous case, where $p_{1} = |\xi|_{g}$ and $E_1 =1$ the condition $(*)$ is clearly automatic. 
 
 
   Since by assumption $\partial_{\xi} p_{1} \neq 0 $ near $p_{1}^{-1}(E_{1})$ we can choose $p_1$ as a local coordinate on $ \pi^{-1}(x)$ near $(x,\xi_{0}) \in p_{1}^{-1}(E_1)$. Then, we put $p_{1} = E_1  r$ and extend it to a local coordinate system $(r,\omega): \pi^{-1}(x) \rightarrow {\mathbb R}^{2}$ near  $(x,\xi_0) \in p_{1}^{-1}(E_1)$.  Cover  a neighbourhood of $\pi^{-1}(x) \cap p_{1}^{-1}(E_1)$ by small open sets and choose a partition of unity subordinate to the covering. Then, make the  change of variables  $(p_{1},\omega) \mapsto \xi$ in each open set and sum over the partition to get
\begin{equation} \label{qci8}
I_{E}(x,\hbar) = (2\pi \hbar)^{-2} \int \int \int \int \exp \, i \, [ 
t(E_1-r) + s (E- p_{2}(x,r \omega)) + {\mathcal O}_{x,\xi}(s^{2}) + {\mathcal O}_{x,\xi}(t^{2})/\hbar] 
\end{equation}
 $$ \times c(x,r \omega,s,t;\hbar) r dr d\omega dt ds + {\mathcal O}(1).$$
where $\omega \in p_{1}^{-1}(E_1) \cap \pi^{-1}(x)$ is a (generalized) angle variable and $d\omega = |\nabla_{\xi}p_{1}(x,\xi)|^{-1} d\xi$ denotes Liouville measure on $p_{1}^{-1}(E_1) \cap \pi^{-1}(x)$. One final application of stationary phase in the $(r,t)$-variables in (\ref{qci8}) gives
\begin{equation} \label{qci9}
I_{E}(x,\hbar) = (2\pi \hbar)^{-1} \int \int  \exp \, i [  s \, ( E- p_{2}(x, \omega)  ) + {\mathcal O}(s^{2}) ] /\hbar] 
\end{equation}
$$ \times c(x, E_1 \omega,s;\hbar)  d\omega  ds + {\mathcal O}(1).$$

The remainder of the proof of Theorem \ref{mainthm} involves integrating the restriction of $I_{E}(x;\hbar)$ in (\ref{qci8}) to $x= x(\tau) \in \gamma$ and then carrying out a detailed analysis of the result under the generic Morse condition in (\ref{generic}).
From (\ref{qci8}),
\begin{equation} \label{mainintegral}
\int_{a}^{b} I_{E}(x(\tau);\hbar) \, d\tau = (2\pi \hbar)^{-1} \int \int \int  \exp \, i [   s \, ( E- p_{2}(x(\tau)  \omega)  )  +{\mathcal O}(s^{2}) ] /\hbar] 
\end{equation}
$$ \times c(x(\tau),  \omega,s;\hbar)  d\omega d\tau ds +  {\mathcal O}(1).$$
\begin{equation} \label{mainintegral2}
=(2\pi \hbar)^{-1} \int  e^{isE/\hbar} \,I_{\gamma}(s;\hbar) \, ds + {\mathcal O}(1).
\end{equation}
In (\ref{mainintegral}) we now absorb the $O(s^{2})$-term into the phase and write  $p_{2}(x(\tau),\omega;s) = p_{2}(x(\tau),\omega) + {\mathcal O}(s),$ uniformly in $(\tau,\omega) \in [a,b] \times S^{1}$. We have also applied Fubini to ensure that the $s$-integral is carried out last since  we want to maintain uniformity in the energy values $E$. By  carrying out  the $s$-integration last, $E$ will always appear in a harmless, linear fashion in the phase {\em only}.  As a result, the uniformity of the estimates for (\ref{mainintegral}) in $E$ is clear.
So, for $\hbar \in (0,\hbar_{0}]$ it remains to estimate the integral:
\begin{equation} \label{mainintegral3}
I_{\gamma}(s;\hbar) =  \int_{C_{\gamma}}  e^{ -i s p_{2}(x(\tau), \omega;s) /\hbar} c(x(\tau),\omega,s;\hbar)  d\omega d\tau .
\end{equation}

Because of the Morse assumption (\ref{generic}), the $(\omega,\tau)$-critical points of $p_{2}(x(\tau),\omega)$ are isolated and so, without loss of generality, we assume that there is a single critical point at $(\tau_{0},\omega_{0})$. Let $B_{0} \subset C_{\gamma}$ be a small neighbourhood of $(\tau_{0},\omega_{0})$. Let $\chi_{j} \in C_{0}^{\infty}(C_{\gamma}); \,\, j=0,1$ be a partition of unity subordinate to  a covering $B_{0} \cup B_1$ of $C_{\gamma}$.
We split up the integral
\begin{equation} \label{split1}
I_{\gamma}(s;\hbar) 
=\int \int e^{-i s p_{2}(x(\tau), \omega;s)  /\hbar} c(x(\tau),\omega,s;\hbar)  \, \chi_{0}(\tau,\omega) \, d\omega d\tau 
\end{equation}
$$ + \int \int e^{-i s p_{2}(x(\tau), \omega;s)  /\hbar} c(x(\tau),\omega,s;\hbar)  \, \chi_{1}(\tau,\omega) \, d\omega d\tau =: I_{\gamma}^{(0)}(s;\hbar) +  I_{\gamma}^{(1)}(s;\hbar).$$
  First, we deal with the second integral  $I_{\gamma}^{(1)}(s;\hbar)$  on the RHS of (\ref{split1}): For $(\tau, \omega) \in $ supp $\chi_{1}$, we have that for $|s|$ sufficiently small,
 \begin{equation} \label{max}
 \max \,\, ( \, |\partial_{\tau} p_{2}( x(\tau), \omega;s)| ,  |\partial_{\omega} p_{2}( x(\tau), \omega;s)| \, ) \geq \frac{1}{C_{0}} >0,
 \end{equation}
 By the implicit function theorem,  in the case where $|\partial_{\omega} p_{2}( x(\tau), \omega;s)| \geq \frac{1}{C_0}$ one can make a local change of variables $(\tau, \omega) \mapsto (\tau, p_{2}(x(\tau),\omega;s) )$ in $I_{\gamma}^{(1)}(s;\hbar)$. Alternatively, when $|\partial_{\tau} p_{2}( x(\tau), \omega;s)| \geq \frac{1}{C_0}$ one can make the make the change of variables $(\tau, \omega) \mapsto (p_{2}(x(\tau),\omega;s), \omega) )$. So, in either case after making a  change of variables, one gets
 \begin{equation} \label{residualintegral}
 (2\pi \hbar)^{-1} \int e^{iEs/\hbar} I_{\gamma}^{(1)}(s;\hbar) ds = (2\pi \hbar)^{-1} \int \int \int e^{is(E-\theta)/\hbar} \tilde{c}_{1}(s,\theta,v;\hbar) \, d\theta dv ds.
 \end{equation}
 where, again $\tilde{c} \in S^{0}_{cl}(1)$ with compact support in all variables. Finally, another application of stationary phase in the $(s,\theta)$-variables gives
 \begin{equation} \label{residual2}
 (2\pi \hbar)^{-1} \int e^{iEs/\hbar} I_{\gamma}^{(1)}(s;\hbar) ds  = {\mathcal O}(1).
\end{equation}
Moreover, the ${\mathcal O}(1)$-bound on the RHS in (\ref{residual2}) is clearly uniform in $E$.

 We now deal with $I_{\gamma}^{(0)}(s;\hbar)$.   The Morse assumption and implicit function theorem imply that he critical point equations 
 $$ \partial_{\tau} p_{2}(x(\tau),\omega;s) =0, \,\,\, \partial_{\omega} p_{2}(x(\tau),\omega;s) = 0,$$
 $$\tau(0) = \tau_0, \,\,\,\, \omega(0) = \omega_0$$
 have unique local solutions $\tau(s)$ and $\omega(s)$ which are smooth for $|s|\leq \frac{1}{C}$ with $C>0$ sufficiently large. We apply stationary phase in $(\tau,\omega)$  to expand the first integral $I_{\gamma}^{(0)}(s;\hbar)$
on the RHS of (\ref{split1}). First, we split up the domain of $s$-integration and write
\begin{equation} \label{split3}
\int_{C_\gamma} e^{is p_{2}(x(\tau),\omega;s)/\hbar} \chi_{0}(\tau,\omega) c(x(\tau),\omega,s;\hbar) \, d\tau d\omega =\int_{C_\gamma} {\bf 1}_{|s| \leq \hbar} e^{is p_{2}(x(\tau),\omega;s)/\hbar} \chi_{0}(\tau,\omega) c(x(\tau),\omega,s;\hbar) \, d\tau d\omega  \end{equation}
$$+ \int_{C_\gamma}  {\bf 1}_{|s|\geq \hbar} e^{is p_{2}(x(\tau),\omega;s)/\hbar} \chi_{0}(\tau,\omega) c(x(\tau),\omega,s;\hbar) \, d\tau d\omega. $$
Clearly,
\begin{equation} \label{easy}
(2\pi \hbar)^{-1} \int_{|s| \leq \hbar} e^{iEs/\hbar} I_{\gamma}^{(0)}(s;\hbar) ds = {\mathcal O}(1).
\end{equation}
An application of stationary phase with parameters (\cite{Ho} Theorem 7.7.5) in the second integral gives
\begin{equation} \label{sp}
 {\bf 1}_{|s| \geq \hbar} I_{\gamma}^{(0)}(s;\hbar) = \hbar s^{-1} \, c _{0}(x(\tau(s)),\omega(s),s)  \, {\bf 1}_{|s| \geq \hbar} \, \exp \left[ \,i s  p_{2}(x(\tau(s)),\omega(s);s)/\hbar \, \right]
 \end{equation}
 $$ + {\mathcal O}(|s|^{-2} \hbar^{2}).$$
 So, integrating (\ref{sp}) over $\{s;  1 \geq |s| \geq \hbar \}$ gives
 \begin{equation} \label{upshot}
\left| (2\pi \hbar)^{-1} \int_{1 \geq |s| \geq \hbar} e^{iEs/\hbar} I_{\gamma}^{(0)}(s;\hbar) ds \right|  \leq C_{1} \hbar^{-1} \int_{1 \geq |s|\ \geq \hbar} \frac{\hbar}{s} ds + C_{2} \hbar^{-1} \int_{1 \geq |s| \geq \hbar}  \frac{\hbar^{2}}{s^{2}} ds
\end{equation}
$$ = {\mathcal O}(|\log \hbar|) + {\mathcal O}(1) = {\mathcal O}(|\log \hbar|). $$
Combining (\ref{residual2}), (\ref{easy}) and  (\ref{upshot}) and using the fact that each of these estimates is uniform in $E$ implies that
for $\hbar \in (0, \hbar_{0}],$
$$ \sup_{ \{ E;  \, (E_{1},E) \in {\mathcal P}(T^*M), \, \}} \int_{a}^{b} I_{E}(x(\tau);\hbar)  \, d\tau = {\mathcal O}(|\log \hbar|).$$
Rewriting this gives the estimate
\begin{equation} \label{trace bound}
 \sup_{ \{ E;  \, (E_{1},E) \in {\mathcal P}(T^*M) \} } \sum_{j} \rho(\hbar^{-1}[\lambda_{j}^{(1)}(\hbar)-E_1] ) \rho(\hbar^{-1}[\lambda_{j}^{(2)}(\hbar)-E]) \times \int_{\gamma} |\phi_{j}^{\hbar}|^{2} ds = {\mathcal O}(|\log \hbar|).
\end{equation}
 We claim that for  any $\hbar \in (0,\hbar_0],$ and  $(\lambda_j^{(1)}(\hbar), \lambda_j^{(2)}(\hbar)) \in \mbox{Spec} (P_1(\hbar), P_2(\hbar)),$ with $|\lambda_j^{(1)}(\hbar) - E_1|\leq C_1\hbar$, there exists $C_2>0$ such that
$$ \inf_{E}   | \lambda_j^{(2)}(\hbar) - E| ) \leq C_2 \hbar.  \,\,\, (*)$$  To see this, we argue by contradiction:  Assume that $(*)$ does not hold.  Then there exists   a sequence $(\hbar_m)_{m=1}^{\infty}$with $\hbar_m \rightarrow 0^+$ as $m \rightarrow \infty$  for which $(*)$ is violated. Let $\epsilon \in (\hbar_m)_{m=1}^{\infty}$ and treat $\epsilon >0$ as an adiabatic parameter. Consider the $\epsilon$-pseudodifferential operator $P(\epsilon):= \epsilon^{-2}  \, [ \, P_{1}(\epsilon) - \lambda_j^{(1)}(\epsilon)  ] ^{2} + \epsilon^{-2} \, [  P_{2}(\epsilon) - \lambda_{j}^{(2)}(\epsilon ) ]^{2}$ and put
  $p_{k,\epsilon} = \epsilon^{-2}  \, ( \, p_k - \lambda_j^{(k)}(\epsilon) \, )^{2} \, $ and $P_{k,\epsilon} = \epsilon^{-2}  \ [  P_k(\epsilon)  - \lambda_j(\epsilon) \, ]^{2}; \, k=1,2.$
     So, then our assumption implies that for any $\epsilon  \in (\hbar_m)_{m=1}^{\infty},$ 
        $$ p_{2,\epsilon}(x,\xi) \geq C_2^{2} \,\, \mbox{when} \,\,  p_{1,\epsilon} (x,\xi) \leq C_1^{2} $$
 and thus  $P(\epsilon)= P_{1,\epsilon} + P_{2,\epsilon} $ is $\epsilon$-elliptic.
  
One then constructs an $\epsilon$-parametrix $Q(\epsilon)$ with
   $$ Q(\epsilon)  P(\epsilon)  = Id + {\mathcal O}(\epsilon^{\infty})_{L^{2} \rightarrow L^{2}}.  $$
    Applying $ Q(\epsilon)$ to both sides of the equation $P(\epsilon) \phi_{j}^{(\epsilon) } = 0$  implies that 
\begin{equation}\label{mass1}
\|  \phi_j^{(\epsilon) } \|_{L^{2}} = {\mathcal O}(\epsilon^{\infty}). 
\end{equation}
But since $\epsilon >0$ can be taken arbitrarily small,  (\ref{mass1})    contradicts the fact that all joint eigenfunctions are $L^{2}$-normalized.

So, after possibly rescaling $\rho$ and using that   $\rho \geq 0$ with $\rho(0) =1$ it follows from $(*)$ that for all $j \geq 1$ and $\hbar \in (0,\hbar_{0}],$  there exists a constant $C_{3}>0$ (independent of $j$) such that
$$\sup_{ \{ E; (E_{1},E) \in {\mathcal P}(T^*M), \, |E_1 - \lambda_j^{(1)}(\hbar)| \leq C_1 \hbar \} } \rho(\hbar^{-1}[\lambda_{j}^{(2)}(\hbar)-E]) \geq C_{3}>0.$$
Since the sum on the LHS of (\ref{trace bound}) has non-negative terms,  by restricting to  $\{ j; |\lambda_{j}^{(1)}(\hbar) - E_{1}| \leq C_{1} \hbar \}$   and  (after possibly rescaling $\rho$) using that 
$\rho(\hbar^{-1}[\lambda_{j}^{(1)}(\hbar) - E_1]) \geq C_5 >0$
 for these eigenvalues, one finally gets that 
  $$  \sum_{ \{j; |\lambda_{j}^{(1)}(\hbar) - E_1| = O(\hbar) \} }  \int_{\gamma} |\phi_{j}^{\hbar}|^{2} ds = {\mathcal O}_{|\gamma|}( |\log \hbar|).$$

This     finishes the proof of Theorem \ref{mainthm}. 
\qed

\begin{remark} The sup bound in Theorem \ref{mainthm} is also uniform in the energy parameter, $E_{1}.$  However, for different values of $E_{1}$ one needs to excise different subvarieties of $M$ (which depend on $E_{1}$) to ensure that $p_{1}$ is real principal type on $p_{1}^{-1}(E_{1})$.  For example, in the case where $p_{1} = |\xi|_{g}^{2} + V(x),$ assumption (A1) requires that $\gamma \cap \{x \in M; V(x) = E_{1} \} = \emptyset.$
\end{remark}
\section{Non-generic curves} \label{nongeneric}
In this section, we turn to the proof of Theorem \ref{mainthm2}. In contrast to Theorem \ref{mainthm}, this results deals with the $L^{2}$-restriction bounds of joint eigenfunctions of $P_{1}(\hbar)$ and $P_{2}(\hbar)$ with $\hbar$-microsupports along singular orbits of the joint bicharacteristic flow of $H_{p_1}$ and $H_{p_2}$. We show that, up to $\log \hbar$-factors, the maximal $L^{2}$-restriction bound in \cite{BGT} is attained along the base projections of these orbits. In the special homogeneous case, these projections are certain (exceptional) geodesics. For example, as we discuss in section \ref{examples}, in the case of surfaces of revolution, the equator is such an exceptional geodesic. It is however the {\em only} exceptional geodesic: all other geodesics including the meridian great circles are generic in the sense of Definition \ref{generic}.

 Just as in the previous section, the analysis boils down to estimating the integral $I_{\gamma}(s;\hbar)$. However, unlike the generic case, the phase function $\Psi(\tau,\omega) \in C^{\infty}_{\gamma}$ will now have degenerate critical points and we will use a change of variables to classical Birkhoff normal form along these singular orbits to compute the asymptotics.

\subsection{Orbits of the joint flow, $\Phi^{t}$.}
Here, we describe an important class of  exceptional curves, $\gamma$, which
 do not satisfy (\ref{generic}).  As we have already pointed out in the introduction, it is not difficult to see that in the homogeneous case, geodesics  are distinguished as far $L^p$-restriction bounds are concerned  (see for example \cite{BGT}). In the QCI case, the same is true for the bicharacteristics of general inhomogeneous Hamiltonians. Moreover, as we will show,  the {\em nature} of bicharacteristics vis-a-vis the singular Lagrangian foliation of $T^{*}M$ also plays a very important role as far restriction bounds are concerned. 
 First, we give a slightly different characterization of what it means for a curve $ \gamma$ to be generic. This consists of a series of simple but important geometric lemmas, the main result being Proposition \ref{mainprop}.  
 
 Fix a smooth curve $\gamma:[a,b] \rightarrow {\mathbb R}$ and let $(\tau(0), \omega(0)) \in C_{\gamma}$ be any point on the cylinder.  We define $ \Psi:C_{\gamma} \rightarrow {\mathbb R}$ by $\Psi = \iota^{*} p_2,$ where
 $\iota: C_{\gamma} \rightarrow T^*M$ is the standard inclusion map.
So, in terms of the local coordinates $(\tau, \omega): C_{\gamma} \rightarrow {\mathbb R}^{2}$, 
$$ \Psi(\tau,\omega) = p_{2}(x(\tau), \omega).$$
Modulo an ${\mathcal O}(s)$-error (which is negligible), this is the phase function in (\ref{mainintegral3}),  in the integral $I_{\gamma}(s;\hbar)$.

 The point  $(\tau(0), \omega(0))$ is  critical for $\Psi: C_{\gamma} \rightarrow {\mathbb R}$ if for every smooth curve segment $\mu(s):=(\tau(s), \omega(s)) \in C_{\gamma}; \, s \in (-\epsilon, \epsilon)$ passing through the initial point,
 \begin{equation} \label{gen1}
 \frac{\partial}{\partial s} \Psi(\tau(s), \omega(s))|_{s=0}  =  0.
 \end{equation}
 Since $\Psi(\tau(s),\omega(s)) = p_{2}(\tau(s),\omega(s))$, writing (\ref{gen1}) out explicitly and applying the chain rule gives:
 \begin{equation} \label{gen2}
 \partial_{x} p_{2} \cdot \partial_{s} \tau|_{s=0}  + \partial_{\xi} p_{2} \cdot \partial_{s} \omega |_{s=0} = 0.
 \end{equation}
 On the other hand, differentiating the definining equation $p_{1}(\tau(s), \omega(s)) = 1$ gives
 \begin{equation} \label{gen3}
 \partial_{x} p_{1} \cdot \partial_{s} \tau|_{s=0}  + \partial_{\xi} p_{1} \cdot \partial_{s} \omega |_{s=0} = 0.
 \end{equation}
  The following lemma is  an immediate consequence of (\ref{gen2}) and (\ref{gen3}).
  \begin{lem} \label{generic2}
 A point $z_{0} = (\tau(0),\omega(0)) \in C_{\gamma}$ is critical for $\Psi: C_{\gamma} \rightarrow {\mathbb R}$ if and only if
 $$ T_{z_0}C_{\gamma} \subset \ker (dp_{1})(z_{0}) \cap \ker (dp_{2})(z_{0}).$$
\end{lem}
The following simple geometric result is central to our proof of Theorem \ref{mainthm2} since it describes the bicharacteristics that are non-generic.
\begin{proposition} \label{mainprop}
Let $\gamma \subset \pi(\tilde{\gamma})$ where $\tilde{\gamma} = \vartheta(z_{0})$ is a joint orbit of $\exp t_{j} H_{p_j}; j=1,2$ through the point $z_{0} \in C_{\gamma}$ with $\dim \tilde{\gamma} \geq 1$.  Then, if $ \tilde{\gamma} \subset C_{\gamma},$ the curve $\gamma$    is {\em not} generic.
\end{proposition}
\begin{proof} First, the real principal type assumption combined with the implicit function theorem imply that $C_{\gamma} = \pi^{-1}(\gamma) \cap p_1^{-1}(E_1)$ is a smooth two-dimensional submanifold of $T^{*}M$. We split the analysis into two cases.

\noindent{\bf Case 1:} When $\tilde{\gamma}$ is a two-dimensional Lagrangian torus, we have that locally
$$\tilde{\gamma} = p_{1}^{-1}(E_1) \cap p_{2}^{-1}(E)$$
for some $E \in {\mathbb R}$. Since by assumption $\tilde{\gamma} \subset C_{\gamma}$, and both are two-manifolds, clearly
$$C_{\gamma} = \tilde{\gamma}.$$
Then, $C_{\gamma}$ is non-generic since $\Psi = p_{2}|_{C_\gamma} = E$ and so, all points $z \in C_\gamma$ are critical for $\Psi$.

\noindent{\bf Case 2:} Here we assume that $\tilde{\gamma}$ is a singular joint orbit of dimension one (see subsection \ref{singleaves} below). Then, for all $z \in \tilde{\gamma},$
$$ dp_{2}(z) = \lambda(z) \cdot dp_{1}(z),$$
for some $\lambda(z) \neq 0$. So, from Lemma \ref{generic2},  $z_{0} \in \tilde{\gamma}$ is a critical point of $\Psi: C_{\gamma} \rightarrow {\mathbb R}$ if and only if
$$ T_{z_0} C_{\gamma} \subset \ker (dp_1)(z_0).$$
This  inclusion is always satisfied since $C_{\gamma} \subset p_{1}^{-1}(E_1)$. As a result, {\em all} points $z \in \tilde{\gamma}$ along the one-dimensial orbit are critical for $
\Psi$ and so the latter is not Morse. \end{proof}
\subsection{Singular leaves of the Lagrangian foliation.} \label{singleaves}
  Before taking up the proof of Theorem  \ref{mainthm2} we collect here some basic facts about the geometry of integrable systems and their singular sets. We refer the reader to \cite{TZ3,VN2} for further details.
  
    Given the moment map ${\mathcal P}= (p_{1}, p_{2})$, the {\em singular variety} of the corresponding integrable system is defined to be the set
   $$\Sigma_{sing} = \{ (x,\xi) \in T^{*}M; \, dp_{1} \wedge dp_{2}(x,\xi) = 0 \}.$$
 We now recall some elementary results about $\Sigma_{sing}$ which we will need later on. 
  First, given the joint flow $\Phi^{t}: T^{*}M \rightarrow T^{*}M$ defined by
 $\Phi^{t}(x,\xi) = \exp t_{1} H_{p_1} \circ \exp t_{2} H_{p_{2}}(x,\xi); \,\,\,\, t=(t_1,t_2) \in {\mathbb R}^{2},$ we observe that 
  \begin{equation} \label{invariance}
  \Phi^{t}(\Sigma_{sing}) = \Sigma_{sing}.
  \end{equation}
 which follows immediately from the fact that $\{ p_1,p_2 \} = 0$ and $\Phi^{t}$ is a diffeomorphism.
The singular set $\Sigma_{sing}$ consists of a union of orbits of the joint flow $\Phi^{t}; t \in {\mathbb R}^{2}$. 

\begin{defn} \label{singorbit}
Following \cite{TZ3}, we say that an orbit $\vartheta$ of the joint flow $\Phi^{t}$   {\em singular} if it is not Lagrangian; that is, if $\dim \vartheta \leq 1$.
\end{defn}
\subsubsection{Eliasson nondegeneracy}
For our second main result (Theorem \ref{mainthm2}), we will need to make a non-degenerate assumption on the integrable system with moment map ${\mathcal P} = (p_1,p_2)$. We now give a brief description of this condition. For more detailed treatment,  see  \cite{VN1, VN2, TZ3}.
Let ${\bf p} = {\mathbb R} \{  p_{1}, p_{2} \} \subset  C^{\infty}(T^{*}M-0), \{. \}$ be the standard abelian subalgebra with Poisson bracket. Then,
 given a singular orbit $\vartheta(v) = \exp t_{1} H_{p_1} \circ \exp t_{2} H_{p_2}(v)$ through a point $v \in {\mathcal P}^{-1}({\mathcal B}_{sing})$ of rank $k \leq 1$, we note that the Hessians $d^{2}_{v}p_{j}; j=1,2,$ determine an Abelian subalgebra
$$ d^{2}_{v} {\bf p} \subset S^{2}(K/L, \omega_v)^{*}$$
of quadratic forms on the reduced symplectic subspace $K/L$, where we put
$$ K = \ker dp_{1}(v) \cap \ker dp_{2}(v), \,\,\,\,\, L =  span ( H_{p_1}(v), H_{p_{2}}(v) ).$$
\begin{defn} \label{Eliasson}
We say that the orbit $\vartheta(v)$ is Eliasson non-degenerate of rank $k \leq 1$ if $d^{2}_{v} {\bf p}$ is a Cartan subalgebra of $S^{2}(K/L, \omega_{v})^{*}$.
\end{defn}

\begin{lem} \label{sing}
Assume that the integrable system with moment map ${\mathcal P}= (p_1,p_2)$ is Eliasson non-degenerate. Then, $\Sigma_{sing}$ is a finite union of  orbits of the joint flow, $\Phi^{t}$ with dimension $\leq 1$. The latter are diffeomorphic to open intervals, circles and  isolated points. 
\end{lem}
\begin{proof} From (\ref{invariance}) it follows that $\Sigma_{sing}$ is a finite union of joint orbits of the joint flow $\Phi^{t}$
and so has dimension $\leq 2$. The Eliasson non-degeneracy condition (\ref{Eliasson}) implies that  $\dim \Sigma_{sing} \leq 1$.  As a result, $\Sigma_{sing}$ consists of a union of open intervals, circles and, in the inhomogeneous case, possibly a finite number of isolated points. 
\end{proof}
\begin{remark} In the homogeneous case where $p_{1}(x,\xi)= |\xi|^{2}_{g}$, the singular orbits are necessarily topological intervals or circles since ${\mathcal P} =(p_1,p_2)$ has no isolated critical points.
\end{remark}

The Eliasson non-degeneracy assumption implies that $\Sigma$ is a finite union of isolated singular orbits for $p_j; j=1,2$ and we  use this in the next section  to analyze  the  integral (\ref{mainintegral3})  by microlocalizing near these orbits and applying a classical Birkhoff normal form construction to analyze the resulting integral.

We note  that the crucial difference between the generic case and the case of a bicharacteristic which lifts to a singular joint orbit lies in the fact that due to the invariance of the integral $p_{2}$,
 the computation of $I(s;\hbar)$ can be reduced to a {\em single} fibre $ \pi^{-1}(x) \cap p_{1}^{-1}(E_1)$ in the latter case. Thus, there is no additional cancellation coming from the computation of the $s$-integral and this is ultimately the reason why the ${\mathcal O}(\hbar^{-1/2}) \, L^{2}$-restriction  bound is saturated by these singular orbits.

 \subsection{Microlocalization along $C_{\gamma}$}
 In the following, it is useful to split up the mapping cylinder $C_{\gamma}$ as follows:
\begin{equation} \label{split}
C_{\gamma} = C_{\gamma}^{reg} \cup C_{\gamma}^{sing},
\end{equation}
where $C^{reg}_{\gamma}$ (resp. $C^{sing}_{\gamma}$) denote invariant open neighbourhoods of regular (resp. singular) points of $C_\gamma$.  Let $\chi_{reg}(\omega)$ (resp. $\chi_{sing}(\omega)$) be a partition of unity subordinate to this covering of the cylinder, $C_{\gamma}$.
We then write
\begin{eqnarray} \label{split2}
I_{\gamma}(s;\hbar) =   \int  \int e^{-i s \Psi(\tau,\omega)  /\hbar} c(\omega,\tau, s;\hbar)  \chi_{reg}(\omega) \, d\omega d\tau \nonumber \\
+  \int  \int e^{-i s \Psi(\tau,\omega)  /\hbar} c(\omega,\tau,s;\hbar) \chi_{sing}(\omega) \, d\omega d\tau =: I_{reg}(s;\hbar) + I_{sing}(s;\hbar).
\end{eqnarray}
First, we analyze the regular term on the RHS of (\ref{split2}).
\subsection{Analysis of the regular term.}
 Given $\omega \in $ supp $\chi_{reg}$, 
 in light of the invariance formula (\ref{invariance}) it easily follows that for all  $\tau \in [a,b],$ and $\omega \in $ supp $\chi_{reg},$
  $$d\Psi (\tau,\omega) = d(\iota^{*}p_{2})(\tau,\omega) \neq 0.$$
Indeed, rank $(dp_{1},dp_{2})(\tau,\omega) =2$ for all $\omega \in $ supp $\chi_{reg}$ and so, by Lagrange multipliers, the restriction $\Psi = \iota^{*}p_{2} \in C^{\infty}(p_1^{-1}(E_0)) $ satisfies $d\Psi(\tau,\omega) = d(\iota^{*}p_{2})(\tau,\omega) \neq 0$ for all $(\tau,\omega) \in [a,b] \times $ supp $\chi_{reg}$. 
  But then one can introduce $\iota^{*}p_{2}$ as a new coordinate on supp  $\chi_{reg}$ and  so by the change of variables formula,
 \begin{equation} \label{generic4}
(2\pi \hbar)^{-1} \int  \int e^{i [ sE -s\theta]/\hbar} c(s,\theta;\hbar) d\theta ds = {\mathcal O}(1).
\end{equation}
The last bound on the RHS of (\ref{generic4}) follows by stationary phase in $(s,\theta)$ and the estimate  is uniform for $ E \in \pi_{2} ( {\mathcal P}(T^{*}M) )$ where $\pi_{2}:(E_{1},E) \mapsto E$.




\subsection{Analysis of the singular term.}
Here we assume that for $\{(x(\tau), \omega); a \leq \tau \leq b \}  \in supp \chi_{sing}$. So, in particular $(x(\tau),\omega)$ is contained in an arbitrarily small neighbourhood of $(x(0),\omega(0)) \in \tilde{\gamma}$ where,
$$ dp_{1} \wedge dp_{2}(x(0),\omega(0)) =0.$$

To deal with the second term in (\ref{split2}), it is useful to pass to a convergent singular Birkhoff normal form  and write the phase function $\Psi(\tau, \omega)$ in (\ref{split2})  in normal coordinates. The analysis will be split into several cases depending on the nature of the singularity in the phase function, $\Psi$.

\subsubsection{Singular Birkhoff normal forms.}
First, we recall  that
the orbit $\vartheta= \cup_{(t_1,t_2) \in {\mathbb R}^{2}} \exp t_1 H_{p_{1}} \circ  \exp t_{2} H_{p_2} (x(0),\omega(0))$ of the joint flow is of dimension $\leq 1$. So, it is diffeomorphic to a union of intervals and circles and possiblty a finite number of (necessarily isolated) critical points. Since the latter case of fixed points is handled very similarily to the case of 1-D orbits, but we only consider here restriction bounds along curves 



The literature on general classical (and quantum) BIrkhoff normal forms  is extensive \cite{G1,G2,ISZ,Z1,Z2} and we focus here on the integrable case where the canonical change of variables to normal form is actually convergent \cite{CP,HS,T2,TZ3, MVN,VN2}.  Without loss of generality, we assume here that the singular locus $\tilde{\gamma} = \Phi_{t}(x(0),\omega(0))$ consists of a bicharacteristic. Whether or  not  $\gamma$  is the projection of a periodic bicharacteristic is of no consequence here. Since we are considering the case where $n=2$, there are only two possibilities: $\gamma$ is either {\em stable} (elliptic) or {\em unstable} (hyperbolic).

\subsubsection{Stable case.} \label{stablecase}
 Let $\gamma$ be a non-degenerate, {\em stable} bicharacteristic in the singular locus $ {\mathcal P}^{-1}(b) \subset {\mathcal P}^{-1}({\mathcal B}_{sing})$.  Let  $(x, t) : M \rightarrow {\mathbb R}^{2}$ be coordinates centered at the point $x_{0} \in M$. For instance one can take $(x, t)$ to be Fermi coordinates along $\gamma$. In this case, the $t$-coordinate runs along the geodesic and the $x$-coordinate is transversal.  By possibly  replacing  $p_{1}$ and $p_2$ by appropriate functions $f_{j}(p_{1},p_{2}); j=1,2$ (the corresponding operators $f_{j}(P_{1}(\hbar), P_{2}(\hbar)); j=1,2$ have the same joint eigenfunctions), one can assume that
 $$p_{j}(x,t,\xi,\sigma) = b_{j} + \delta_{j}(\sigma) + \omega_{j}(\sigma)(x^{2} + \xi^{2}) + {\mathcal O}_{t,\sigma}( |x,\xi|^{3}); \,,\,\, j=1,2,$$ 
where, $\delta_j(0) = 0, \, \omega_{j}(0) \neq 0; \, j=1,2$ and $\omega_j, \delta_j$ are locally-defined smooth functions near $\sigma =0$.

 In this case, \cite{TZ3, VN1, VN2} there exists a canonical map from a small neighbourhood $U_{\gamma}$ of $C_{\gamma}$
 $$ \kappa: U_{\gamma} \longrightarrow T^{*}\gamma \times  B_\delta(0),$$
 $$ \kappa : (x,t;\xi,\sigma) \mapsto (x',t';\xi',\sigma'),$$
 with \begin{equation} \label{normalform1}   
 \kappa^{*}p_{j}(x',t';\xi',\sigma') = F_{j}(x'^{2} +\xi'^{2}, \sigma'); \,\,\,\,j=1,2.
 \end{equation}
 Here, $\delta >0$ is a sufficiently small tube radius and $F_{j} \in C^{\infty}( B_{\delta}(0) \times B_{\delta}(0))$. By possibly replacing the classical integrals $p_{j}; j=1,2$ by $f_{k}(p_{1},p_{2}); k=1,2$ with appropriate $f_{k} \in C^{\infty},$ without loss of generality, we can assume that 
 $$ F_{j} (u,v) =  b_{j} + \beta_{j}(v) + \alpha_j(v) u + {\mathcal O}_{v}(u^2 ).$$
 Moreover, one can take here $\beta_{j}(v)  = v + {\mathcal O}(v^{2})$.
 The Eliasson non-degeneracy condition says that for all $v \in B_\delta(0), \,\,\alpha_{1} \neq \alpha_{2}$ with $\min \{ |\alpha_{1}|, |\alpha_{2}| \} \geq \frac{1}{C} >0.$
 We need to compute the asymptotics of the RHS of (\ref{split2}). Without loss of generality, one can assume that $x_{0} \in M$ is an interior point of the segment $\gamma$ and so $a<0, b>0$.
 \begin{eqnarray} \label{normal form2}
  I_{sing}(s;\hbar) = \int \int_{ C_{\gamma} } e^{-is \Psi(t,\omega) / \hbar} \chi_{sing}(\omega) c'(\omega,t;s) \, d\omega dt.
  \end{eqnarray}
  To make the change of variables in (\ref{normal form2}) to Birkhoff coordinates $(x',t';\xi',\sigma') \in T^{*}(\gamma) \times B_{\delta}(0)$, we use that
  $$ x'(x,t;\xi,\sigma) = x + {\mathcal O}(x^{2}), \,\,\, t'(x,t;\xi,\sigma) = t + {\mathcal O}(x),$$
  and so, from the expansion in (\ref{normalform1}), one gets that
  \begin{equation} \label{normalform3}
  \kappa (p_{1}^{-1}(E_1) \cap \pi^{-1}(\gamma)) =  \{ (x',t';\xi',\sigma'); \, x'=0, \,\,  \sigma' + \alpha_1(\sigma')   \xi'^2 + {\mathcal O}_{\sigma}(\xi'^{4}) + {\mathcal O}(\sigma'^2) = 0 \}.
  \end{equation}
To simplify the writing, from now on we drop the primes in the Birkhoff coordinates and put $(b_1,b_2) = (E_1,E)$.  Then, since
$$ \frac{\partial}{\partial \sigma} \left( \sigma + \alpha_1(\sigma)   \xi^2 + + {\mathcal O}(\sigma^{2}) +  {\mathcal O}_{\sigma}(\xi^{4})  \right) = 1 + {\mathcal O}(\sigma) + {\mathcal O}( \xi^{2}) $$
it follows from the implicit function theorem that one can use   $(\sigma,t) \in B_{\delta}(0) \times \gamma$ as  local parametrizing coordinates on $\kappa (p_{1}^{-1}(E_1) \cap \pi^{-1}(\gamma)) = \kappa (C_{\gamma})$. Substitution of the defining equation in (\ref{normalform3}) into the formula for $\kappa^{*} p_{2}$ gives
  $$\Psi(\sigma) = E + \beta_{2}(\sigma) + \alpha_{2}(\sigma)  \xi^{2}(\sigma) + {\mathcal O}(\xi^{4}(\sigma))$$
  $$=  E + \sigma -\sigma  \frac{\alpha_{2}(\sigma)}{\alpha_{1}(\sigma)}   + {\mathcal O}(\sigma^{2}) $$
  \begin{equation}\label{phase}
  = E + \left( 1 -  \frac{\alpha_{2}(\sigma)}{\alpha_{1}(\sigma)}  \right)  \sigma + {\mathcal O}(\sigma^{2}).
  \end{equation}
 Here we have used that $\beta_{2}(\sigma) = \sigma + {\mathcal O}(\sigma^{2})$ with $\beta_{2} \neq 0$.
 
 Next we compute the induced measure $d\omega$ in terms of the Birkhoff coordinates. 
  Let $\Omega$ denotes the canonical 2-form locally given by
  $dx\wedge d\xi + dt \wedge d\sigma$. Since $\kappa$ is canonical,  locally the Lebesgue measure
 $$  (\kappa^*\Omega)^{2} = \Omega^{2} = dx dt d\xi d\sigma.$$
The induced arc-length (ie. Liouville measure) $d\omega$ satisfies 
\begin{equation} \label{arclength}
 \kappa_{*}d\omega dt=  i_{*}   d\sigma dt.
\end{equation}
In (\ref{arclength}),  $i: \gamma \times  \mbox{supp} \chi_{1}  \rightarrow \kappa (C_{\gamma})$ is the local parametrization given by
$$i(t,\sigma) = (0,\xi(\sigma), t,\sigma),$$
where $\xi(\sigma)$ satisfies the identity in (\ref{normalform3}).

Choosing $(\sigma,t)$ as   coordinates on supp $\chi_{1} \times \gamma$, by a straightforward computation we get that 
\begin{equation} \label{measure}
\kappa_{*} d\omega dt =   f(\sigma) |\sigma|^{-1/2}  \, d\sigma dt,
\end{equation}
where $ f \in C^{\infty}$ with  $f(\sigma) \geq \frac{1}{C} >0.$ 
 Consequently, by a change of variables, in terms of the normal coordinates we get that
  \begin{equation} \label{upshot1}
  (2\pi \hbar)^{-1} \int e^{isb_{2}/\hbar} I_{sing}(s;\hbar) ds = (2\pi \hbar)^{-1} \int  \int \int  \exp is  \left( \left( 1 -  \frac{\alpha_{2}(\sigma)}{\alpha_{1}(\sigma)}  \right)  \sigma + {\mathcal O}(\sigma^{2})  \right)
 \end{equation} 
$$ \times  c(s,\sigma,t;\hbar)\chi_{sing}(\sigma) |\sigma|^{-1/2} d \sigma dt  ds.$$
where $c \in S^{0}(1) \cap C^{\infty}_{0}$ is $\hbar$-elliptic on supp $\chi_{sing}$. Now, from the non-degeneracy of the integrable system, we use the fact that $\alpha_{2}(\sigma) \neq \alpha_{1}(\sigma)$ to make the change of variables
$$\sigma \mapsto \left( 1 -  \frac{\alpha_{2}(\sigma)}{\alpha_{1}(\sigma)}  \right)  \sigma + {\mathcal O}(\sigma^{2}) $$
in the phase in (\ref{upshot1}) and integate out the $t$-variable (note that the phase in (\ref{upshot1}) is independent of $t$). The result is that
\begin{equation} \label{homogeneous}
  (2\pi \hbar)^{-1} \int e^{isb_{2}/\hbar} I_{sing}(s;\hbar) ds =(2\pi \hbar)^{-1} \int  \int    e^{is \sigma/\hbar}  \tilde{c}(s,\sigma;\hbar)  |\sigma|^{-1/2} d \sigma   ds.
  \end{equation}
  Here the amplitude $\tilde{c}$ has the same properties as $c$.

Making a first-order Taylor expansion around $\sigma = 0$ we write $ \tilde{c}(s,\sigma;\hbar) = \tilde{c}(s,0;\hbar) + \sigma  \cdot \delta \tilde{c} (s,\sigma;\hbar)$ where $\delta \tilde{c}(s,\sigma;\hbar) \in S^{0}(1)$ with compact support in the $s$-variable and both $\tilde{c}$ and $\delta$ have standard symbolic expansions in $\hbar$ with $\delta \tilde{c}(s,\sigma;\hbar) = \delta \tilde{c} (s,\sigma) + {\mathcal O}(\hbar)$ and $\tilde{c}(s,0;\hbar) = \tilde{c}(s,0) + {\mathcal O}(\hbar)$.  Assume that supp $\chi_{sing}  \subset \{ \sigma; |\sigma| \leq C \}.$ The integral in (\ref{upshot1})  splits into the sum
\begin{equation} \label{upshot3}
(2\pi \hbar)^{-1} \int  \int  e^{is \sigma/\hbar}  \tilde{c}(s,0)  |\sigma|^{-1/2} {\bf 1}_{|\sigma| \leq C} (\sigma)  d \sigma ds 
\end{equation}
$$+ (2\pi \hbar)^{-1} \int  \int  e^{is \sigma/\hbar}  \delta \tilde{c} (s,\sigma)  |\sigma|^{1/2}{\bf 1}_{|\sigma| \leq C} (\sigma)  d \sigma  ds + {\mathcal O}(1)=: I_{sing}^{(1)}(\hbar) + I_{sing}^{(2)}(\hbar) + {\mathcal O}(1).$$
Let ${\mathcal F \phi}(\xi) = (2\pi)^{-n} \int_{{\mathbb R}^{n}} e^{-ix\xi} \phi(x) dx$ be the usual Fourier transform.  Then, by Fubini,
$$I_{sing}^{(2)}(\hbar) = (2\pi \hbar)^{-1} \int_{|\sigma| \leq C} |\sigma|^{1/2} ({\mathcal F \delta \tilde{c}})_{s\rightarrow \sigma}(\hbar^{-1} \sigma, \sigma) d\sigma = {\mathcal O}(\hbar^{1/2}). $$
On the other hand, again by Fubini, for the leading term 
$$I_{sing}^{(1)}(\hbar) = (2\pi \hbar)^{-1} \int_{|\sigma| \leq C}  \left( \int  e^{is \sigma /\hbar}  \tilde{c}(s,1)  ds \right)  |\sigma|^{-1/2} d \sigma $$
\begin{equation} \label{upshot4}
  = (2\pi \hbar)^{-1} \int  ({\mathcal F \tilde{c}})_{s \rightarrow \sigma}(\hbar^{-1} \sigma,1)  \,   |\sigma|^{-1/2} d \sigma \sim_{\hbar \rightarrow 0} c_{\gamma} \hbar^{-1/2}.
  \end{equation} Again, the constant $c_{\gamma}>0$ appearing on the RHS in (\ref{upshot4}) is uniform in $E$ with $(E_1,E)\in {\mathcal P}(T^{*}M).$ Consequently,
  $$ (2\pi \hbar)^{-1}\int e^{-is E /\hbar} I(s;\hbar) ds = I_{sing}^{(1)}(\hbar) + I_{sing}^{(2)}(\hbar) + I_{reg}(\hbar) $$
  $$= c_{\gamma} \hbar^{-1/2} + {\mathcal O}(\hbar^{1/2}) + {\mathcal O}(1) = c_{\gamma} \hbar^{-1/2} + {\mathcal O}(1).$$
  This completes the proof of Theorem \ref{mainthm2}. \qed
  
 
\subsubsection{Unstable case.} 
 In this case, the relevant canonical transformation to normal form is given by 
 $ \kappa : U_{\gamma} \rightarrow T^{*} \gamma \times B_{1}(0),$
 where,
 $$\kappa^{-1*} p_{j} (x,t;\xi,\sigma) = F(\xi^{2} - x^{2}, \sigma) = b_{j} + \beta_{j}(\sigma)+ \alpha_{j}(\sigma) (\xi^{2}-x^{2}) + {\mathcal O}(|\xi^{2}-x^{2}|^{2}); \,\,\,\, j=1,2.$$
 The computations follow in the same way as in the stable case by putting $x=0$ and repeating essentially verbatim the analysis in \ref{stablecase}. 




 From (\ref{upshot4}) it follows that 
 \begin{equation} \label{singestimate}
   \sum_{j} \rho( \hbar^{-1} [\lambda_{j}^{(1)}(\hbar) - E_{1}]) \,  \rho( \hbar^{-1} [\lambda_{j}^{(2)}(\hbar) - E]) \,\int_{\gamma} |\phi_{j}^{\hbar}(s)|^{2} ds \sim_{\hbar \rightarrow 0}  c_{\gamma}(E;\rho) \hbar^{-1/2}.
 \end{equation}
 So, by taking supremum over $E$ in (\ref{singestimate}) the first estimate in (ii) of Theorem \ref{mainthm2} follows.
The final part of the proof of Theorem \ref{mainthm2} follows from the result of Toth and Zelditch \cite{TZ3} which says that, unless $(M,g)$ is a flat torus, the bicharacteristic flow must have a singular orbit, $\tilde{\gamma}$.  But then, $\tilde{\gamma} \subset {\mathcal P}^{-1}(E_{1},E)$ where $(E_{1},E) \in {\mathcal B}_{sing}$. 
 In the case where $\tilde{\gamma}$ is stable, the existence of the subsequence of joint eigenfunctions follows from the usual joint trace formula 
 \begin{equation} \label{jointtrace}
   \sum_{j} \rho( \hbar^{-1} [\lambda_{j}^{(1)}(\hbar) - E_{1}]) \,  \rho( \hbar^{-1} [\lambda_{j}^{(2)}(\hbar) - E])  \sim_{\hbar \rightarrow 0} c(E;\rho).
  \end{equation}
 To see this, one simply argues by contradiction: Assume that for all eigenfunctions
 $$ \int_{\gamma} |\phi^{\hbar}_{j}(s)|^{2} ds = o(\hbar^{-1/2}).$$
 Then we bound the LHS in (\ref{singestimate}) by
 $$ o(\hbar^{-1/2})  \times \sum_{j} \rho( \hbar^{-1} [\lambda_{j}^{(1)}(\hbar) - E_{1}]) \,  \rho( \hbar^{-1} [\lambda_{j}^{(2)}(\hbar) - E]) = o(\hbar^{-1/2})$$
 by the joint trace formula (\ref{jointtrace}). This contradicts the asymptotic $\sim \hbar^{-1/2}$ on the RHS of (\ref{singestimate}).
 
 In the unstable case \cite{BPU, TZ3}, the  formula (\ref{jointtrace}) gets replaced by
 $$  \sum_{j} \rho( \hbar^{-1} [\lambda_{j}^{(1)}(\hbar) - E_{1}]) \,  \rho( \hbar^{-1} [\lambda_{j}^{(2)}(\hbar) - E])  \sim_{\hbar \rightarrow 0} c(E;\rho) | \log \hbar |.$$
 In view of (\ref{singestimate}), this gives a subseqeunce $\phi_{j_k}^{\hbar}; k=1,2,3,...$ satisfying 
 $$ \int_{\gamma} |\phi_{j_k}^{\hbar}(s)|^{2} ds \geq   c_{\gamma} \hbar^{-1/2} |\log \hbar|^{-1}$$
for $\hbar \in (0,\hbar_{0}]$.  This completes the proof of Theorem \ref{mainthm2}. \qed

\section{The example of a convex surface of revolution.} \label{examples}


One can parametrize convex surfaces of revolution by using geodesic polar coordinates $(t,\phi) \in [0.,1] \times [0,2\pi]$ in terms of which
 $$p_{1}(\phi,t;\xi_{\phi}, \xi_{t}) = \xi_{t}^{2} + a^{-1}(t) \, \xi_{\phi}^{2}, $$
 and
 $$p_{2}(\phi,\theta;\xi_{\phi}, \xi_{\theta}) = \xi^{2}_{\phi},$$
 where,  the profile function satisfies $a(0)=a(1)=0$ and $a(t)$ is a non-negative Morse function with a single non-degenerate maximum at $t = t_{0} \in (0,1)$. The level curve $t= t_{0}$ is the equator of the surface. Let $\gamma = \{ (t, \phi(t)); 0 < a \leq t \leq b < 1\}$ be a curve segment on the surface.
 The computation of the phase function $\Psi$ in this case is easy. Clearly,
 $$ C_{\gamma} = \{ (t, \phi(t); \xi_{t},\xi_{\phi}); \xi_{t}^{2} + a^{-1}(t) \xi_{\phi}^{2} = 1 \}$$
 and  one can use $t \in (0,1)$ and $\xi_{t}$ to parametrize $C_{\gamma}$. The result is that
 $$ \Psi(t, \xi_{t}) = a(t) \, ( 1 - \xi_{t}^{2}); \,\,\,\, t \in [a,b].$$
 The critical points are the solutions of
 $$ \partial_{t} \Psi = a'(t) (1-\xi_{t}^{2}) = 0 \,\,\, \mbox{and} \,\,\,\partial_{\xi_t} \Psi = -2 a(t) \xi_{t} = 0.$$
  Since $t \in (0,1)$, there is a single critical point with $\xi_{t} = 0$ and $a'(t) = 0$.  This happens precisely when $t = t_{0}$.  The end result is that the critical point of $\Psi$  is
 $(t_{0}, 0).$ 
 

We next compute the terms of the Hessian matrix at the critical point $(t_0,0)$.  The result is that  $\partial_{t}^{2} \Psi = a''(t_0),  \, \partial_{t} \partial_{\xi_t}\Psi=0$ and
 $\partial_{\xi_t}^{2} \Psi = - 2 a(t_{0}).$ Consequently, we get that
 $$ \det ( d^{2} \Psi )|_{ (t_{0},0)}  = -2 a(t_0) a''(t_{0}) \neq 0$$ by our Morse assumption on the profile function of the surface.
 Thus,  it follows that the curve segment $\gamma$ is generic. These curves are segments of graphs over the meridian great circle.

Next consider curves of the form 
$$ \gamma = \{ (\theta(t), \phi(t) = t);  \, t \in  [a,b] \}$$
which are graphs over the equator.  In this case, 
$$\Psi(t,\xi_{t}) = a(\theta(t)) (1-\xi_{t}^{2}); \,\,\, t \in [a,b].$$
  The critical points in this case are the solutions of
  $$ \partial_{t} \Psi = a'(\theta(t)) \cdot \theta'(t) (1-\xi_{t}^{2}) = 0 \,\,\, and \,\,\, \partial_{\xi_t} \Psi = -2 a(\theta(t)) \xi_{t} = 0.$$
  Since $a(\theta) >0$ the second equation implies that $\xi_{t} =0$ at critical points. The first equation implies
  $$a'(\theta(t)) =0\,\, \mbox{or} \,\,\, \theta'(t) = 0.$$
 For the Hessian,
  $$ \partial_{t}^2 \Psi = [ a''(\theta(t)) | \theta'(t)|^{2} + a'(\theta(t)) \theta''(t)] (1-\xi_t^{2}) = 0, $$
  and $$ \partial_{\xi_t}^{2} \Psi = -2 a(\theta(t)).$$
  Also, clearly $\partial_{t} \partial_{\xi_t} \Psi = 0$ at any critical point $(t_0,0)$.  
  In the case where $a'(\theta(t_0))=0$ at the crtical point $(t_0,0)$
   one gets
  $$\det  (d^{2} \Psi)|_{(t_0,0)} = -2 a(\theta(t_0)) \,  a''(\theta(t_0))  \, | \theta'(t_0)|^{2}. \,\,\,\, (*)$$
    In the case where $\theta'(t_0) = 0$ at the critical point $(t_0,0)$ one gets
   $$\det  (d^{2} \Psi)|_{(t_0,0)} = -2 a(\theta(t_0)) \,  a'(\theta(t_0))  \,  \theta''(t_0) \,\,\,\, (**).$$ The only way $(*)$ can vanish is if also $\theta'(t_0) =0$, so that both $\theta'(t_0)=0$ and $a'(\theta(t_0))=0$; that is, the curve $\gamma(t)$ is tangent to the equator at $t=t_0$.
    In the second case where $\theta'(t_0)=0$ and $a'(\theta(t_0)) \neq 0$, the curve $\gamma(t)$ is tangent to another circle parallel to the equator.

So, curves which are graphs over the equator are generic in the sense of Definition \ref{generic} provided  $\theta'(t) \neq 0$. This condition is satisfied provided  $\theta :[a,b] \rightarrow (0,\pi)$ is never tangent to a  circle parallel to the equator.  In particular, this rules out the cases where $\gamma$ includes pieces of the equator $ z = 0$  or parallel circles $z = const.$. The equator is of course the (only) projection of a singular orbit. It is non-generic and in that case,  Theorem \ref{mainthm2} applies. The parallel circles $z= const.$ are caustics which are also necessarily non-generic since in the latter case there are joint eigenfunctions which blow-up like $\sim \lambda^{1/6}$ in $\sup$-norm
  along the curve.
  
   To see what Theorem \ref{mainthm} means for a specific sequence of eigenfunctions, we consider the special case of the round sphere where $\phi_{\lambda} (x) = \lambda^{1/4} (x_{1} + i x_{2})^{\lambda}$ are the  highest-weight spherical harmonics where $\int_{M} |\phi_{\lambda}|^{2} d\vol  \sim 1$ and where $\lambda = n; \, n=1,2,3,...$. In the previous paragraph we showed that
   all smooth curves $\gamma = \{ (\theta(t) = t, \phi(t) );  \, t \in [a,b] \}$ are generic. In terms of spherical coordinates, the restricted eigenfunction  
   $$\phi_{\lambda}(t) = \lambda^{1/4} [ \cos t \cos \phi(t) + i \cos t \sin \phi(t) ]^{\lambda}$$
   and so,
   $$\int_{\gamma} |\phi_{\lambda}|^{2} ds = \lambda^{1/2} \int_{a}^{b} (\cos t)^{2\lambda} dt .$$
   In the case where $a <0$ and $b >0$ (so that $\gamma$ intersects the equator), an application of steepest descent gives
  $$ \lambda^{1/2} \int_{a}^{b} (\cos t)^{2\lambda} dt  \sim_{\lambda \rightarrow \infty} c_{\gamma} = {\mathcal O}(1).$$
  Similarily, when $\gamma = \{ (\theta(t),t); t \in [a,b] \}$ with $a<0, b>0$ and  $|\theta'(t)| \geq \frac{1}{C} >0$ one gets
  $$ \lambda^{1/2} \int_{a}^{b} (\cos \theta(t))^{2\lambda} dt  \sim_{\lambda \rightarrow \infty} \tilde{c}_{\gamma} = {\mathcal O}(1).$$
  These bounds are   consistent  with (and slightly stronger than)  the  general ${\mathcal O}(\log \lambda)$ bound given in Theorem \ref{mainthm}.
\subsection{Zonal harmonics.}  \label{zonal}
Let $x=(x_1,x_2)$ be geodesic normal coordinates on a convex surface of revolution centered at
 the north pole and $(r,\phi)$ denote the corresponding polar variables. We consider zonal harmonics centered at the north pole which can be written as oscillatory integrals  of the
 form
 \begin{equation} \label{zonal1}
 \phi_{\lambda}(x) = (2\pi \lambda)^{1/2} \int_{{\mathbb S}^{1}} e^{i \lambda \langle x, \omega \rangle}  a(x,\omega;\lambda) \, d\omega,
\end{equation}
where, $a(x,\omega;\lambda) \sim \sum_{j=0}^{\infty} a_{j}(x,\omega) \lambda^{-j}$ and $|a_{0}(x,\omega)| \geq \frac{1}{C} > 0$ with $a_{0}(x,\omega) = 1 + {\mathcal O}(|x|)$.    
The $\lambda^{1/2}$-factor in front of the terms in (\ref{zonal1}) ensures that
 $ \int_{M} |\phi_{\lambda}|^{2} dx = 1$. Conisider the meridian great circle
 $$\gamma = \{ (r,\phi); \phi = \alpha_0; \, 0 \leq \alpha_{0} \leq 2\pi \}.$$
From (\ref{zonal1}) it follows that $\phi_{\lambda}$ is radial and we get that
$$\int_{0}^{\pi} |\phi_{\lambda}(r,\alpha_0)|^{2} dr = 2\pi \lambda \int_{r=0}^{\lambda^{-1} } \left| \int_{{\mathbb S}^{1}}  e^{i\lambda \langle x, \omega \rangle } a(x,\omega;\lambda) d\omega  \right|^{2} dr + 2\pi \lambda \int_{r= \lambda^{-1}}^{\pi} \left| \int_{{\mathbb S}^{1}} e^{i \lambda \langle x, \omega \rangle} a(x,\omega;\lambda) d\omega \right|^{2} dr$$
$$= 2\pi \lambda \int_{r= \lambda^{-1}}^{\pi} \left| \int_{{\mathbb S}^{1}} e^{i \lambda \langle x, \omega \rangle} a(x,\omega;\lambda) d\omega \right|^{2} dr  + {\mathcal O}(1).$$
An application of stationary phase in the inner integral on the RHS of the last identity gives
$$\int_{0}^{\pi} |\phi_{\lambda}(r,\alpha_0)|^{2} dr  = 2\pi  \int_{\lambda^{-1}}^{\pi} \left| r^{-1/2} e^{i\lambda r} \right|^{2} \, dr + {\mathcal O}(1) = 2\pi \int_{\lambda^{-1}}^{\pi} \frac{dr}{r} + {\mathcal O}(1).$$
So, it follows that
\begin{equation} \label{zonal2}
\int_{0}^{\pi} |\phi_{\lambda}(r,\alpha_0)|^{2} dr  = 2\pi  \log \lambda + {\mathcal O}(1),
\end{equation}
and this example shows that the upper bounds in Theorems \ref{mainthm} and \ref{corollary} are sharp.

%

\end{document}